\documentclass[12pt]{amsart}    
\usepackage{latexsym,mathtools}
\usepackage{epsfig,setspace}
\usepackage{a4}
\usepackage{amssymb}
\usepackage{amsthm}
\usepackage{amsbsy}
\usepackage{amsfonts,amssymb,latexsym,amscd,amsmath,euscript,enumerate,verbatim,calc}
\allowdisplaybreaks[1]
\usepackage{url}
\usepackage{hhline}
\usepackage{pifont}
\usepackage[colorlinks=true,linkcolor=blue,citecolor=magenta]{hyperref}

\theoremstyle{plain}

\newtheorem{theorem}{Theorem}[section]

\theoremstyle{definition}

\newtheorem{remark}[theorem]{Remark}

\def\F{{\mathbb F}}

\def\Fq{{\mathbb F}_q}

\def\Fqk{{\mathbb F}_{q^k}}

\newcommand{\Tr}{{\rm Tr}}

\makeatletter
\def\imod#1{\allowbreak\mkern10mu({\operator@font mod}\,\,#1)}
\makeatother

\title{A New Proof of Fitzgerald's Characterization of Primitive Polynomials}

\author{Samrith Ram}
\address{Institut de Math\'{e}matiques de Luminy \newline \indent 
Luminy Case 907 \newline \indent
13288 Marseille Cedex 9 \newline \indent
France \newline \indent}
\email{samrith@gmail.com}



\begin{document}
\maketitle
\begin{abstract}
We give a new proof of Fitzgerald's criterion for primitive polynomials over a finite field. Existing proofs essentially use the theory of linear recurrences over finite fields. Here, we give a much shorter and self-contained proof which does not use the theory of linear recurrences. 
\end{abstract}
\section{Introduction}
Fitzgerald \cite{Fitz} gave a criterion for distinguishing primitive polynomials among irreducible ones by counting the number of nonzero coefficients in a certain quotient of polynomials. Subsequently, Laohakosol and Pintoptang \cite{LaPi} modified and extended the result of Fitzgerald using similar techniques and appealing to the theory of linear recurrences. Here, we prove Fitzgerald's original result by a more direct approach using elementary properties of the trace map.
\section{Fitzgerald's Theorem}

\begin{theorem}[Fitzgerald]
\label{fitzgerald}
  Let $p(x)\in \Fq[x]$ be a monic irreducible polynomial of degree $k$ with $p(1) \neq 0$. Let $m=q^k-1$ and define $g(x)=(x^m-1)/(x-1)p(x)$. Then $p(x)$ is primitive iff $g(x)$ is a polynomial with exactly $(q-1)q^{k-1}-1$ nonzero terms.
\end{theorem}
\begin{proof}
If $p(0)=0$ then $p(x)$ cannot be primitive. So suppose $p(0)\neq 0$. Then $g(x)$ is a polynomial of degree at most $m-1$. Let $q(x)$ be the monic reciprocal of $p(x)$ and let $q(x)=(x-\alpha_1)\cdots (x-\alpha_k)$ be the factorization of $q(x)$ in $\Fqk[x]$. Then
\begin{equation*}
p(x)=a\prod_{i=1}^{k}(1-\alpha_i x)   
\end{equation*}
for some $a\in \Fq^*$. We then have the partial fraction decomposition
\begin{equation*}
  \frac{1}{p(x)}=\frac{1}{a}\sum_{i=1}^{k}\frac{a_i}{1-\alpha_i x},
\end{equation*}
where $a_i=\alpha_i^{k-1}/q'(\alpha_i)$ for $1 \leq i \leq k$. Expanding each term of the partial fraction formally as a power series and collecting terms, we obtain
\begin{equation*}
  \frac{1}{p(x)}=\frac{1}{a}(s_{k-1}+s_kx+s_{k+1}x^2+\cdots),
\end{equation*}
where
\begin{equation*}
  s_r=\sum_{i=1}^{k}\frac{\alpha_i^r}{q'(\alpha_i)}=\Tr\left(\frac{\alpha^r}{q'(\alpha)}\right)
\end{equation*}
for each integer $r$ and $\alpha=\alpha_1$. Here, $\Tr : \Fqk \to \Fq$ is the trace map. Now, we have 
\begin{align*}
  g(x)&=\frac{x^m-1}{(x-1)p(x)}=\frac{1}{a}(1+x+\cdots x^{m-1})(s_{k-1}+s_kx+s_{k+1}x^2+\cdots).
\end{align*}
Since we already know that $g(x)$ is a polynomial of degree less than $m$ we compute the coefficient of $x^t$ in the above product for $0\leq t \leq m-1$. This coefficient equals
\begin{align*}
  \sum_{i=k-1}^{k+t-1}s_i&=\Tr\left(\sum_{i=k-1}^{k+t-1}\frac{\alpha^r}{q'(\alpha)}\right)\\
     &=\Tr\left(\frac{\alpha^{k-1}(1-\alpha^{t+1})}{q'(\alpha)(1-\alpha)}\right)\\
     &=\Tr(\beta)-\Tr(\beta \alpha^{t+1})
\end{align*}
where $\beta=\alpha^{k-1}/q'(\alpha)(1-\alpha)$. Thus the number of nonzero coefficients in $g(x)$ is equal to the cardinality of %the number of nonnegative integers $t\leq m-1$ such that $\Tr(\beta)-\Tr(\beta \alpha^{t+1})$ is nonzero.
\begin{equation*}
\left\{t: \Tr(\beta)-\Tr(\beta \alpha^{t+1})\neq 0, 0 \leq t \leq m-1\right\}.
\end{equation*}
We claim that $\Tr(\beta)\neq 0$. To see this, note that 
\begin{equation*}
\label{traceofb}
  \Tr(\beta)=\sum_{i=1}^{k}\frac{\alpha_i^{k-1}}{(1-\alpha_i)q'(\alpha_i)}
\end{equation*}
To compute $\Tr(\beta)$, let $y$ be an indeterminate and consider the Lagrange interpolation polynomial for $y^{k-1}$ at $\alpha_1,\ldots,\alpha_k$:
\begin{align*}
  y^{k-1}&=\sum_{i=1}^{k}\frac{\alpha_i^{k-1}}{q'(\alpha_i)}\prod_{j\neq i}(y-\alpha_j)\\
   &=\sum_{i=1}^{k}\frac{\alpha_i^{k-1}}{q'(\alpha_i)}\frac{q(y)}{(y-\alpha_i)}
\end{align*}
Setting $y=1$ we find that $\Tr(\beta)=1/q(1)\neq 0$ (since $p(1)\neq 0$), proving the claim.

  Suppose $p(x)$ is primitive. Then $q(x)$ is also primitive and consequently $\alpha$ has multiplicative order $q^k-1$ in $\Fqk^*$. Therefore, the set $\{\beta \alpha^{t+1}:0\leq t \leq m-1\}$ is precisely $\Fqk^*$. Since the map $\Tr: \Fqk\to \Fq$ is surjective and all its fibers have the same cardinality, there are precisely $q^{k-1}$ values of $t$ ($0\leq t \leq m-1$) for which $\Tr(\beta \alpha^{t+1})=\Tr(\beta)$ (since $\Tr(\beta)\neq 0$). Thus, the number of nonzero coefficients of $g(x)$ in this case is $m-q^{k-1}$.

For the converse, suppose $p(x)$ is not primitive. Then neither is $q(x)$ and hence, the multiplicative order (say $e$) of $\alpha$ is a proper divisor of $q^k-1$. In this case, the number (say $N$) of values of $t$ ($0\leq t \leq m-1$) for which $\Tr(\beta \alpha^{t+1})=\Tr(\beta)$ is an integer multiple of $(q^k-1)/e$. Since $(q^k-1)/e>1$ and $(q^k-1)/e$ is coprime to $q^{k-1}$, it follows that $N \neq q^{k-1}$. Thus, the number of nonzero coefficients of $g(x)$ cannot be $m-q^{k-1}$. This completes the proof.   
\end{proof}

\begin{remark}
The condition $p(1) \neq 0 $ is imposed to rule out the polynomial $p(x)=x-1$ which is primitive in $\F_2[x]$.
\end{remark}
\bibliographystyle{abbrv}
\bibliography{/home/samrith/Dropbox/math/bibliography/mybib}

\begin{thebibliography}{1}

\bibitem{Fitz}
R.~W. Fitzgerald.
\newblock A characterization of primitive polynomials over finite fields.
\newblock {\em Finite Fields Appl.}, 9(1):117--121, 2003.

\bibitem{LaPi}
V.~Laohakosol and U.~Pintoptang.
\newblock A modification of {F}itzgerald's characterization of primitive
  polynomials over a finite field.
\newblock {\em Finite Fields Appl.}, 14(1):85--91, 2008.

\end{thebibliography}

\end{document}